\documentclass[12pt]{amsart}

\usepackage{frenchineq}
\usepackage{amsaddr}
\usepackage[colorinlistoftodos,bordercolor=orange,backgroundcolor=orange!20,linecolor=orange,textsize=scriptsize]{todonotes}

\usepackage[utf8]{inputenc}
\usepackage{amsfonts}
\usepackage{amsmath}
\usepackage{amsthm}
\usepackage{amssymb}
\usepackage{mathtools}

\usepackage{caption}
\usepackage{subcaption}

\usepackage{paralist}

\usepackage{centernot}

\usepackage[colorinlistoftodos,bordercolor=orange,backgroundcolor=orange!20,linecolor=orange,textsize=scriptsize]{todonotes}

\DeclareMathOperator{\R}{\mathbb{R}}
 
\DeclareMathOperator{\rk}{rk}
\DeclareMathOperator{\Ker}{Ker} 
\DeclareMathOperator{\Image}{Im}

\DeclareMathOperator{\Span}{Span}

\DeclareMathOperator{\ch}{ch}
\DeclareMathOperator{\sh}{sh}

\DeclareMathOperator{\Gr}{Gr}

\newcommand{\mat}{w}

\DeclarePairedDelimiter{\norm}{\,\|}{\|\,}
\newcommand{\comp}{\circ}
\newcommand{\Sc}{\mathcal{S}} 
\newcommand{\Oc}{\mathcal{O}}
\renewcommand{\Mc}{\mathcal{M}}
\newcommand{\mlb}{\mathcal{I}} 
\newcommand{\Qc}{\mathcal{Q}}
\newcommand{\Ec}{\mathcal{E}} 
\newcommand{\Bc}{\mathcal{B}} 
\newcommand{\Uc}{\mathcal{U}} 
\newcommand{\Vc}{\mathcal{V}} 
\newcommand{\Wc}{\mathcal{W}} 
\newcommand{\GL}{\mathrm{GL}}

\newcommand{\notimplies}{\centernot\implies}

\newcommand{\bigslant}[2]{{\raisebox{.2em}{$#1$}\left/\raisebox{-.2em}{$#2$}\right.}}

\newtheorem{theorem}{Theorem}[section]
\newtheorem{lemma}[theorem]{Lemma}
\newtheorem{proposition}[theorem]{Proposition}
\newtheorem{corollary}[theorem]{Corollary}
\theoremstyle{definition}
\newtheorem{problem}[theorem]{Problem}
\newtheorem{definition}[theorem]{Definition}
\theoremstyle{remark}
\newtheorem{remark}[theorem]{Remark}

\begin{document}
\title{Maximal lower bounds in the L\"owner order}
\author{Nikolas STOTT}
\address{INRIA and CMAP, \'Ecole polytechnique, CNRS, Universit\'e Paris-Saclay}
\email{nikolas.stott@inria.fr}
\thanks{The author was 
partially supported by the ANR projects {CAFEIN, ANR-12-INSE-0007} and MALTHY, ANR-13-INSE-0003, by ICODE and by the academic research chair ``Complex Systems Engineering'' of \'Ecole polytechnique - THALES - FX - DGA - DASSAULT AVIATION - DCNS Research - ENSTA ParisTech - T{\'e}l{\'e}com ParisTech - Fondation ParisTech - FDO ENSTA and by the PGMO programme of EDF and FMJH}
\subjclass[2010]{Primary 47A63, Secondary 15A63, 06F20, 81Q10. Keywords: L\"owner order, Riesz spaces, antilattices, indefinite orthogonal groups, ellipsoids}






\begin{abstract}
We show that the set of maximal lower bounds of two symmetric matrices
with respect to the L\"owner order can be identified to the quotient set
$O(p,q)/(O(p)\times O(q))$. Here, $(p,q)$ denotes the inertia of the
difference of the two matrices, $O(p)$ is the $p$-th orthogonal group,
and $O(p,q)$ is the indefinite orthogonal group arising from a
quadratic form with inertia $(p,q)$.
We also show that a similar result holds for positive semidefinite maximal lower bounds with maximal rank of two positive semidefinite matrices.
We exhibit a correspondence between the maximal lower bounds $C$ of two matrices $A,B$ and certain pairs of subspaces, describing the directions on which the quadratic form associated with $C$ is tangent to the one associated with $A$ or $B$.
The present results refines a theorem from Kadison that characterizes the existence of the infimum of two symmetric matrices and a theorem from Moreland, Gudder and Ando on the existence of the positive semidefinite infimum of two positive semidefinite matrices.
\end{abstract}


\maketitle

\section{Introduction}

The L\"owner partial order is a basic notion in matrix theory~\cite{Loewner,bhatiaPDM}. It describes the pointwise ordering of real quadratic forms. These forms constitute an ordered vector space that is not a lattice, meaning that two quadratic forms may have two uncomparable minimal upper bounds, or dually, two uncomparable maximal lower bounds.
A classical result by Kadison shows that it is an antilattice, meaning that greatest lower bounds exist only in trivial cases:
\begin{theorem}[Kadison, see~\cite{Kadison51}]
\label{thm:kadison}
Two symmetric matrices cannot have a greatest lower bound in the L\"owner
order unless they are comparable in this order.
\end{theorem}
We refer to the work of Kalauch, Lemmens, and van Gaans~\cite{KLG}
for a recent approach to Kadison's theorem and generalizations in the
setting of Riesz spaces. 
Lower bounds of symmetric matrices have also been extensively studied in the setting of \emph{quantum observables}~\cite{Ando, Gudder, HCQ}, where the main motivation is 
the uniqueness of a positive semidefinite maximal lower bound.
Moreland and Gudder have solved this problem in~\cite{Gudder}.
Their result has been generalized to any pair of positive semidefinite bounded self-adjoint operators by Ando~\cite{Ando}. His proof involved the notion of \emph{generalized short}, which  in the finite dimensional case is defined for positive semidefinite matrices $X,Y$ by
\begin{align*}
[Y]X = \max \{ Z \mid 0 \preceq Z \preceq X,\; \Image Z \subseteq \Image Y \} \,.
\end{align*}
Their results show that the uniqueness of a positive semidefinite maximal lower bound is decided by the comparability of such generalized shorts:
\begin{theorem}[Moreland and Gudder, Ando, see~\cite{Gudder,Ando}]
\label{thm:gudder}
Two positive semidefinite matrices $A$ and $B$ cannot have a unique positive semidefinite maximal lower bound unless the generalized shorts $[A]B$ and $[B]A$ are comparable.
\end{theorem}

The aforementionned theorems raise the issue of characterizing the whole set
of maximal lower bounds of two symmetric matrices $A$ and $B$. Our first main result (Theorem~\ref{thm:main_theo}) shows that this set
can be identified to the quotient space
\begin{align*}
O(p,q)/(O(p)\times O(q))
\end{align*}
where $(p,q)$ denote the inertia of $A-B$, $O(p)$ denotes the
$p$-th orthogonal group, 
and $O(p,q)$ is the indefinite orthogonal group arising from a quadratic form with inertia $(p,q)$. 
It follows that the set of maximal lower bounds is of dimension $pq$.


When $p+q=n$, the dimension of the set of maximal lower bounds is $p(n-p)$ which coincides with the dimension of the Grassmannian $\Gr(n,p)$.
This suggests to look for a parametrization of maximal lower bounds by $p$-dimensionnal subspaces of $\R^n$.
Our second main result (Theorem~\ref{thm:sec_theo}) leads to such a parametrization.
We study more generally the following problem:
given subspaces $\Uc$ and $\Vc$, parametrize the set of maximal lower bounds $C$ of two symmetric matrices $A$ and $B$ that satisfy $\Uc \subseteq \Ker (B-C)$ and $\Vc \subseteq \Ker (A-C)$. We give geometric conditions for the existence of solutions in terms of $\Uc, \Vc$ and the indefinite quadratic form $A-B$, and, if these conditions are met, a parametrization of the set of solutions showing that this set is of dimension $(p-\dim \Uc)(q-\dim \Vc)$.

These results have a geometric consequence that will be dealt with in Section~\ref{sec:appli}.
When specialized to positive definite forms, 
the L\"owner order corresponds to the inclusion order of ellipsoids, up to a reversal. 
We deduce from Theorem~\ref{thm:main_theo} that given 
an ellipsoid $\Ec_C$ minimally enclosing two ellipsoids $\Ec_A, \Ec_B$,
the set of tangency points of $\Ec_C$ with $\Ec_A$ (resp. $\Ec_B$) spans
the kernel of $A-C$ (resp. $B-C$). Moreover, the sum of these kernels must span the whole space. 

Although the present results are stated for real quadratic forms, they carry over to hermitian forms, up to immediate changes.

Let us finally point out some applied motivations of the present work. 
The question of selecting a minimal ellipsoid containing two given ellipoids, or a maximal ellipsoid contained in their intersection, appears in a number of applied fields, 
including optimization~\cite{Ben-Tal:2001:LMC:502969}, control~\cite{Lawson2006}, reachability analysis of dynamical systems~\cite{elltoolbox}, program verification~\cite{AGGPS}, information geometry and mathematical morphology~\cite{angulo,Bur05}. 
In most of the applications dealt with there, an important issue is to select effectively a remarkable minimal upper bound or maximal lower bound.
 For instance, selections arising from minimum or maximal volume considerations (L\"owner's ellipsoids~\cite{ball}) are frequently used. 
We expect the present characterization of the set of all maximal lower bounds 
to lead to more flexibility in 
some of these applications.

\section{Notation}

In the sequel, $\Mc_{p,q}$ denotes the set of $p \times q$ (real) matrices, $\Sc_n$ denotes the set of $n \times n$ symmetric matrices and $A^T$ denotes the transpose of a matrix $A$.
The kernel of $A$ is denoted by $\Ker A$, its range by $\Image A$ and its rank by $\rk A$. 
We denote the orthogonal complement of a subspace $V$ with respect to the standard scalar product by $V^\perp$.
The group of $n \times n$ invertible matrices is denoted by $\GL_n$.
If $A\in \Sc_n$, the {\em inertia}
of $A$ is the triple $(p,q,r)$, where $p$ (resp.\ $q,r$) is the number of positive (resp.\ negative, zero) eigenvalues of $A$, counted with multiplicities.
Given two square matrices $A,B$, the \emph{direct sum} of those matrices, denoted $A \oplus B$, is the block diagonal matrix with blocks $A$ and $B$:
\begin{align*}
A \oplus B = 
\begin{pmatrix}
A & \\
& B
\end{pmatrix} \,.
\end{align*}
We denote by $J_{p,q,r}$ the canonical bilinear form 
of inertia $(p,q,r)$ on $\R^{p+q+r}$. It is defined by 
\begin{align*}
J_{p,q,r} (x,y)= \sum_{i=1}^{p} x_iy_i - \sum_{i=p+1}^{p+q} x_iy_i \,,
\end{align*}
and the corresponding matrix in the canonical basis of $\R^n$ is $I_p \oplus (-I_q) \oplus 0_r$, where $I_n$ (resp. $0_n$) denote the identity matrix (resp. zero matrix) of size $n \times n$.
When $r = 0$, we use the notation $J_{p,q}$ and we denote by $\Oc(p,q)$ the associated (indefinite) orthogonal group of square matrices $S$ such that $S J_{p,q} S^T = J_{p,q}$.
When $q=r=0$, $\Oc(p,q)$ becomes the standard orthogonal group $\Oc(p)$.

The set of symmetric matrices is endowed with the \textit{Löwner order} $\preceq$, which is such that 
\begin{align*}
A \preceq B \iff   \forall x \in \R^n,\; x^T Ax \leq x^TBx  \enspace .
\end{align*}
We write $A \prec B$ when $x^T Ax < x^TBx$ holds for all $x\neq 0$. 
The set of positive semidefinite matrices, denoted $\Sc_n^+$, is the set of matrices $A$ such that $A \succeq 0$.
We say that the matrix $M$ is positive definite (resp. negative definite) over a subspace $\Vc$ if $x^TMx > 0$ (resp. $x^TMx < 0$) holds for all nonzero vectors $x$ in $\Vc$.
Given a positive semidefinite matrix $M$, the square root of the matrix $M$ is the unique positive semidefinite matrix, denoted $M^{1/2}$, such that $M^{1/2}M^{1/2} = M$.

\section{Parametrization of the set of maximal lower bounds of two symmetric matrices}

\subsection{Statement of the main theorem}

Our first main result
is a parametrization of
 the set of maximal lower bounds of two symmetric matrices with respect to the L\"owner order. It implies that this set 
is of dimension $pq$ and that it can be identified with 
$\Oc(p,q)/ \big(\Oc(p) \times \Oc(q) \big)$,
the quotient set of the indefinite orthogonal group $\Oc(p,q)$ by the maximal compact subgroup $\Oc(p) \times \Oc(q)$.
\begin{theorem}
\label{thm:main_theo}
Let $A,B,C\in \Sc_n$ be such that $C \preceq A$ and $C \preceq B$,
and let $(p,q,r)$ denote the inertia of $A-B$. 
The following statements are equivalent:
\begin{enumerate}[(i)]
\item C is a maximal lower bound of $A$ and $B$
\item\label{it:2} $\Ker ( A-C ) + \Ker ( B-C ) = \R^n$
\item\label{it:3} $\rk ( A-C ) = p$ and $\rk ( B-C ) = q$
\item\label{it:4}For all $P \in \GL_n$ revealing the inertia of $A-B$, i.e.
such that $A-B = P J_{p,q,r}P^T$, there exists a unique $M \in \Mc_{p,q}$ such that:
\begin{align*}
	C & = A - PS 
	( I_p \oplus 0_{q} \oplus 0_{r} )	
		S P^T
\quad \text{with}
	\\
	S & =
		\begin{pmatrix}
				( I_p + MM^T ) ^ {1/2} & M \\
				M^T & ( I_q + M^TM ) ^ {1/2}
		\end{pmatrix} \oplus 0_r \enspace .
\end{align*}
\end{enumerate}
\end{theorem}


\begin{remark}
Assertion~\eqref{it:4} can also be rewritten in terms of the matrix $B$:
\begin{align*}
C = B - PS (0_p \oplus I_q \oplus 0_r )
		S P^T
\end{align*}
\end{remark}

\begin{remark}
A similar theorem holds for minimal upper bounds, in which case~\eqref{it:2} is unchanged, while~\eqref{it:3} and~\eqref{it:4} read :
\begin{enumerate}[(i)]
	\setcounter{enumi}{2}
	\item $\rk ( A-C ) = q$ and $\rk  (B-C ) = p$,
	\item $C = A + PS
	 (0_p \oplus I_q \oplus 0_r )
	SP^T
	=
	B +
	PS
	 (I_p \oplus 0_{q} \oplus 0_{r} )
	S P^T$
\end{enumerate}
\end{remark}

Before proving Theorem~\ref{thm:main_theo}, we
draw two corollaries. Theorem~\ref{thm:main_theo}, Corollary~\ref{cor:homeo} and Corollary~\ref{cor:homeo_pos} are proved in Section~\ref{sec:proof_main_theo}.

\begin{corollary}\label{cor:homeo}
Let $A,B \in \Sc_n$, and let $(p,q,r)$ denote the inertia of $A-B$.
Then, the set of maximal lower bounds of $A$ and $B$ is homeomorphic to
the quotient set 
\begin{align*}
\bigslant{\Oc(p,q)}{\big(\Oc(p) \times \Oc(q)\big)} \cong \R^{pq} \enspace .
\end{align*}
\end{corollary}

\begin{corollary}\label{cor:homeo_pos}
Let $A,B \in \Sc_n^+$, and let $(p',q',r')$ denote the inertia of $B[A] - A[B]$. 
The rank of a positive semidefinite maximal lower bound of $A,B$ cannot exceed $p'+q'+\dim \Ker (A-B)$. Moreover,
 the set of \emph{positive semidefinite} maximal lower bounds of $A$ and $B$ which have this rank is homeomorphic to
the quotient set 
\begin{align*}
\bigslant{\Oc(p',q')}{\big(\Oc(p') \times \Oc(q')\big)} \cong \R^{p'q'} \enspace .
\end{align*}
\end{corollary}

We note that Kadison's result can be recovered as a special case of Corollary~\ref{cor:homeo}.
Indeed, the existence of greatest lower bound of two matrices
$A,B$ is equivalent to the existence
of a unique maximal lower bound of these matrices, which, by Corollary~\ref{cor:homeo}, cannot happen unless $pq=0$, meaning that $A\preceq B$ or $B\preceq A$.

The result from Moreland and Gudder can be recovered from Corollary~\ref{cor:homeo_pos} in the same way. If two positive semidefinite matrices $A,B$ have a unique positive semidefinite maximal lower bound $C$, then the uniqueness implies that $p'q' = 0$, which means that $[B]A \preceq [A]B$ or $[A]B \preceq [B]A$.

\subsection{Preliminary lemmas}

We present two results which will be useful in the proof of Theorem~\ref{thm:main_theo}.
\begin{lemma}
\label{intersectionKernelLemma}
Let $P,Q \in \Sc_n$. 
We have
\begin{align*}
\Ker P + \Ker Q = \R^n \implies \Ker P \cap \Ker Q = \Ker ( P - Q )
\end{align*}
\end{lemma}
\begin{proof}
The inclusion $\Ker P \cap \Ker Q \subseteq \Ker ( P - Q )$ is trivial. Let $x \in \Ker ( P-Q )$ and assume $x \notin \Ker P$. 
Then $Px = Qx \neq 0$, and so $\Image P \cap \Image Q \neq \{ 0 \}$. Taking the orthogonal complement contradicts $\Ker P + \Ker Q = \R^n$.
\end{proof}

\begin{lemma}[Polar decomposition of $\Oc(p,q)$, see~\cite{Gallier}]
\label{thm:gallier}
For every $S \in \Oc(p,q)$, there exists a unique triple $(M,U,V) \in \Mc_{p,q} \times \Oc(p) \times \Oc(q)$ such that:
\begin{align*}
S =
\begin{pmatrix}
	\big( I_p + MM^T \big)^{1/2} & M \\
	M^T & \big( I_q + M^TM \big)^{1/2}
\end{pmatrix}
\begin{pmatrix}
U & \\
& V
\end{pmatrix}
 \enspace .
\end{align*}
\end{lemma}

\subsection{Proof of Theorem~\ref{thm:main_theo}, Corollary~\ref{cor:homeo} and Corollary~\ref{cor:homeo_pos}}
\label{sec:proof_main_theo}

We now prove Theorem~\ref{thm:main_theo}.
We shall prove
\begin{align*}
(i) \iff (ii) \iff (iii) \;\;\text{and}\;\; (ii) \iff (iv) \,.
\end{align*}

\subsubsection*{$\neg(i) \implies \neg(ii)$}~\\
If $C$ is not a maximal lower bound of $A$ and $B$, then there is some $C^* \in \Sc_n$ such that :
\begin{align*}
		C \precneqq C^* \,, \quad
		C^* \preceq A \,, \quad
		C^* \preceq B \,.
\end{align*}
Let $x\in \Ker (A-C)$. We have $0 \preceq A-C^* \preceq A-C$, thus $x^T(A-C^*)x = 0$ and by nonnegativity, $Ax = C^*x$. By assumption, $Ax = Cx$, so that $x \in \Ker (C-C^*)$. This shows that $\Ker (A-C) \subseteq \Ker (C-C^*)$. By symmetry, this is also true when changing $A$ into $B$. As a consequence of this inclusion and the assumption $C \neq C^*$, we have:
\begin{align*}
\Ker (A-C) + \Ker (B-C) \subseteq \Ker (C-C^*) \neq \R^n \,.
\end{align*}

\subsubsection*{$\neg(ii) \implies \neg(i)$}~\\
If the sum of kernels is not equal to $\R^n$, then the set $\Ker (A-C)^\perp \cap \Ker (B-C)^\perp$ is not $\{ 0 \}$. Let $u$ be a unit vector in $\Ker (A-C)^\perp \cap \Ker (B-C)^\perp$.
Let $z \in \R^n$. We write $z = x+y$, with $x \in \Ker A-C$ and $y \in  \Ker( A-C )^\perp$.
Then we have, for $\epsilon > 0$,
\begin{align*}
z^T (A-C - \epsilon uu^T)z = y^T(A-C-\epsilon uu^T ) y\,.
\end{align*}
The quadratic map $y \mapsto y^T(A-C)y$ is positive definite on $\Ker (A-C)^\perp$ as the matrix $A-C$ is positive semidefinite, so the quadratic map $z \mapsto z^T(A-C-\epsilon uu^T)z$ is nonnegative over $\R^n$ for $\epsilon$ small enough.
By symmetry, this is also true when changing $A$ into $B$, so that for $\epsilon$ small enough, we have
$	A  \succeq  C + \epsilon uu^T\,, 
	B  \succeq  C + \epsilon uu^T \,,$
and thus $C$ is not a maximal lower bound.

\subsubsection*{$(ii) \implies (iii)$}~\\
We know from Lemma~\ref{intersectionKernelLemma} that $\Ker (A-B) = \Ker (A-C) \cap \Ker (B-C)$. We choose $K_A$ to be a direct summand of $\Ker (A-B)$ in $\Ker (A-C)$. Similarly, we choose $K_B$ a direct summand of $\Ker (A-B)$ in $\Ker (B-C)$. Recall that $(p,q,r)$ denotes the inertia of $A-B$. We have, $x^T(A-B)x = x^T(A-C)x > 0$ for all nonzero $x \in K_B$, so from the definition of the inertia, $\dim K_B \leq p$ and by symmetry, $\dim K_A \leq q$. Those inequalities are equalities since
\begin{align*}
n = \dim \R^n = \dim K_A + \dim K_B + \dim \Ker ( A-B ) \leq q + p + r = n \,.
\end{align*}
We conclude using the rank-nullity theorem: since $\rk (A-C) = n - \dim \Ker (A-C)$, we have $\rk (A-C) = n - q - r = p$. Similarly, we have $\rk (B-C) = q$.

\subsubsection*{$(iii) \implies (ii)$}~\\
We always have $\Ker ( A-C ) \cap \Ker ( B-C ) \subseteq \Ker (A - B )$, and so
\begin{align*}
\dim \big( \Ker ( A-C ) + \Ker ( B-C ) \big)
& = 2n-p-q \\& - \dim \big( \Ker ( A-C ) \cap \Ker ( B-C ) \big) \\
& \geq 2n-p-q-(n-r) \\
& = n \,.
\end{align*}

\subsubsection*{$(ii) \implies (iv)$}~\\
Without loss of generality, we may assume that $P = I_n$, so that $A-B = J_{p,q,r}$. As before, we can write $\R^n = K_A \oplus K_B \oplus \Ker ( A- B ) $ with
\begin{align*}
	\Ker (A-C) = K_A \oplus \Ker ( A-B ) \,, \;
	\Ker (B-C) = K_B \oplus \Ker ( A-B ) \,.
\end{align*}
We build a basis of $\R^n$ respecting this subspace decomposition: we take a basis $\Bc_A$ of $K_A$, a basis $\Bc_B$ of $K_B$ and $\Bc_{A-B}$ of $\Ker (A-B)$, and our basis of $\R^n$ is $\big[ \Bc_B \,;\, \Bc_A \,;\, \Bc_{A-B} \big]$. In this basis, the matrices of the quadratic forms $A-C$ and $B-C$ are block-diagonal matrices:
\[
A-C =
(M_p \oplus 0_q \oplus 0_r)
\;\;\text{and}\;\;
B-C = (0_p \oplus M_q \oplus 0_r )
 \,,
\]
where the off-diagonal blocks are zero, because the matrices $M_p$ and $M_q$ (respectively of size $p\times p$ and $q\times q$) are positive definite.
The matrix $\Sigma  = M_p^{1/2} \oplus M_q^{1/2}$ is in the indefinite orthogonal group $\Oc(p,q)$, since $\Sigma J_{p,q} \Sigma^T = J_{p,q}$. By Lemma~\ref{thm:gallier}, there is a unique tuple $(M,U,V) \in \Mc_{p,q} \times \Oc(p) \times \Oc(q)$ such that :
\[
\Sigma = 
\begin{pmatrix}
	( I_p + MM^T ) ^ {1/2} & M \\
	M^T & ( I_q + M^TM ) ^ {1/2}
\end{pmatrix} 
\begin{pmatrix}
	U & \\
	& V
\end{pmatrix}  \,.
\]
The matrix $M$ does not depend on the choice of the matrix $\Sigma$: it is easily shown that all matrices $\Xi$ such that $\Xi(\Xi^T) =
M_p \oplus M_q$
 only differ from $\Sigma$ by a block-diagonal orthogonal-block matrix :
\[
\Xi = \Sigma  ( U' \oplus V' )
\;,\;
U' \in \Oc(p), \; V' \in \Oc(q) \,,
\]
and any block-diagonal orthogonal-block matrix of this form multiplied on the right vanishes when computing $C$. Indeed, if we denote
\begin{align*}
S =
\begin{pmatrix}
	( I_p + MM^T ) ^ {1/2} & M \\
	M^T & ( I_q + M^TM ) ^ {1/2}
\end{pmatrix}
\oplus 0_r
\; \text{and} \;
W =
U \oplus V \oplus 0_r
\,,
\end{align*}
we have
$
C = A - (SW)
(I_p \oplus 0_{q} \oplus 0_r)
(SW)^T
= A - S
(I_p \oplus 0_{q} \oplus 0_r)
S
\,.
$

\subsubsection*{$(iv) \implies (ii)$}~\\
Without loss of generality, we may again assume that $P = I_n$. 
After change of basis with the invertible matrix
$
Q :=
S^{-1} + 
( 0_{p} \oplus 0_q \oplus I_r )
$,
we have
$
QAQ^T = 
(I_p \oplus 0_q \oplus 0_{r} )
$ and 
$QBQ^T = 
( 0_p \oplus I_q \oplus 0_r )
$.
The sum of the kernels of those matrices is $\R^n$, thus this is also the case for $A$ and $B$.
This concludes the proof of Theorem~\ref{thm:main_theo}. \hfill $\qed$
\begin{proof}[Proof of Corollary~\ref{cor:homeo}]
We have shown in the proof $(ii) \implies (iv)$ of Theorem~\ref{thm:main_theo} that we can associate to every matrix $\Sigma \in \Oc(p,q)$ a maximal lower bound $C$ of $A$ and $B$, via the continuous map $\Phi$ from $\Oc(p,q)$ to $\Sc_n$ defined by:
\begin{align*}
\Phi : \Sigma \mapsto A - P
\begin{pmatrix}
\Sigma & \\ & 0_r
\end{pmatrix}
\begin{pmatrix} I_p \oplus 0_q & \\ & 0_{r} \end{pmatrix}
\begin{pmatrix}
\Sigma & \\ & 0_r
\end{pmatrix}^T
 P^T \,.
\end{align*}
Moreover, we have previously shown that two matrices $\Sigma_1, \Sigma_2 \in \Oc(p,q)$ produce the same maximal lower bound $C$ if and only if $\Sigma_1 = \Sigma_2
(U \oplus V)$ for some matrices $U \in \Oc(p), V\in \Oc(q)$. This proves that the map $\Phi$ is a bijection from $\Oc(p,q)/ \big(\Oc(p) \times \Oc(q)\big)$ to the set of maximal lower bounds of $A,B$. By Lemma~\ref{thm:gallier}, the quotient set can be identified to $\Mc_{p,q} \cong \R^{pq}$ by means of the continuous bijection $S$ defined by:
\begin{align*}
S : M \mapsto 
\begin{pmatrix}
( I_p + MM^T ) ^ {1/2} & M \\
M^T & ( I_q + M^TM ) ^ {1/2}
\end{pmatrix} \,.
\end{align*}
It remains to show that the map $\Phi \circ S$ has a continuous inverse. We write
\begin{align*}
\Phi \circ S(M) = A - P
\bigg[
\begin{pmatrix}
A(M) & B(M) \\
B(M)^T & M^TM  \\
\end{pmatrix}
\oplus 0_r \bigg]
P^T \,,
\end{align*}
with $A(M) := I + MM^T$ and $B(M) := (I+MM^T)^{1/2}M$. 
The matrix $M$ can be recovered continuously with $M = A(M)^{-1/2} B(M)$, since $A(M) \succeq I$ cannot vanish.
\end{proof}

\begin{proof}[Proof of Corollary~\ref{cor:homeo_pos}]
Before treating the general case, we shall prove the corollary when $A,B$ are positive definite. Note that, in this case, we have $[A]B = B$ and $[B]A = A$.

First, the inertias of the matrices $A-B$ and $B^{-1} - A^{-1}$ are the same:  the matrices $A,B$ can be reduced simultaneously by an invertible congruence $X \mapsto PXP^T$ to diagonal matrices with positive diagonal elements $a_i$ and $b_i$. The fact that $a_i- b_i> 0$ is equivalent to $b_i^{-1} - a_i^{-1} >0$ shows that the inertias are identical. Also, note that $p'+q'+\dim \Ker (A-B) = n$, so matrices that have this rank are invertible.

Since the map $X \mapsto X^{-1}$ is monotonically decreasing on the set of positive definite matrices, it is a (continuous) bijection between
the set of minimal upper bounds of $A^{-1},B^{-1}$ (which are positive definite definite) and the set of positive definite maximal lower bounds of $A,B$.
By Corollary~\ref{cor:homeo}, the former set is homeomorphic to $\Oc(p',q') / \big(\Oc(p') \times \Oc(q')\big) \cong \R^{p'q'}$, so the same is true for the latter set.

Now let $A,B$ denote positive semidefinite matrices.
We may assume that $\Ker (A-B) = \{0\}$, since it does not influence the structure of the set of maximal lower bounds of $A,B$ by Theorem~\ref{thm:main_theo}, so that $\Ker A \cap \Ker B = \{0\}$.

Let $R_{A,B}$ denote the set $\Image A \cap \Image B$.
We claim that there are direct summands $R_A$ and $R_B$ of $R_{A,B}$ in $\Image A$ and $\Image B$ respectively so that the matrices of the quadratic forms $A,B$ are block-diagonal in $\R^n = R_A \oplus R_{A,B} \oplus R_B$.

Indeed, we have $\R^n = \Ker B \oplus R_{A,B} \oplus \Ker A$. In such a decomposition, the quadratic forms $A,B$ have matrices of the form
\begin{align*}
A = \begin{pmatrix}
A_{11} & A_{12} \\
A_{12}^T & A_{22} 
\end{pmatrix} \oplus 0_b
\quad\text{and}\quad
B = 0_a \oplus \begin{pmatrix}
B_{22} & B_{23} \\
B_{23}^T & B_{33}
\end{pmatrix} \,.
\end{align*}
We define the matrix $U$ mapping $w = (x,y,z) \in \Ker B \oplus R_{A,B} \oplus \Ker A$ to $Uw = (x -A_{11}^{-1} A_{12}y,y,z - B_{33}^{-1} B_{23}^Ty)$.
One can easily check that the subspaces $R_A$ and $R_B$ defined as the image of $\Ker B$ and $\Ker A$ respectively by $U$ satisfy the desired condition.
Moreover, up to a transformation $X \mapsto V^TXV$ with $V$ block-diagonal, we may assume that 
%
$A = I_a \oplus S_A \oplus 0_b$ and
$B = 0_a \oplus S_B \oplus I_b$,
where $S_A,S_B$ denote positive definite matrices such that $S_A - S_B = J_{p',q'}$.
Note that the short $[B]A$ (resp. $[A]B$) is given by $0_a \oplus S_A \oplus 0_b$ (resp. $0_a \oplus S_B \oplus 0_b$).

Let $C$ denote a positive semidefinite maximal lower bound of $A,B$. Using the characterization in Theorem~\ref{thm:main_theo}, $C$ is given in block form by
\begin{align*}
C = 
\begin{pmatrix}
-XX^T - YY^T & * & * \\
* &* &* \\
*& * &- Y^TY - W^TW
\end{pmatrix}
\end{align*}
%
with $M =
\begin{psmallmatrix}
X & Y \\
Z & W
\end{psmallmatrix}
\in \Mc_{a+p',b+q'}$.

The fact that $C$ is positive semidefinite implies that $X,Y,W$ are zero matrices, so that $C = 0_a \oplus S_C \oplus 0_b$, where the matrix $S_C$ is given by
\begin{align*}
S_C = S_A - 
\begin{pmatrix} 
I_{p'} + ZZ^T & (I_{p'}+ZZ^T)^{1/2}Z \\ 
Z^T(I_{p'}+ZZ^T)^{1/2} & Z^TZ 
\end{pmatrix} \,.
\end{align*}
By Theorem~\ref{thm:main_theo}, $S_C$ is
a (positive semidefinite) maximal lower bound of the matrices $S_A$ and $S_B$ . This concludes the proof since $S_A,S_B$ are positive definite.
\end{proof}


\section{Maximal lower bounds selection under tangency constraints}
\label{mlbstc}

\subsection{Notation and preliminary lemma}

We first give some notation that will be useful in the sequel.
We define the linear operators $\pi_p$, $\pi_q$ and $\pi_r$, mapping respectively $\R^n$ to $\R^p$, $\R^q$ and $\R^r$, that select the first $p$ coordinates, the following $q$ and the last $r$ coordinates. Their matrices in the canonical basis of $\R^n$ are
\[
\pi_p =
\begin{pmatrix}
I_p & 0_{pq} & 0_{pr}
\end{pmatrix},
\qquad
\pi_q =
\begin{pmatrix}
0_{qp} & I_q & 0_{qr}
\end{pmatrix},
\qquad
\pi_r =
\begin{pmatrix}
0_{rp} & 0_{rq} & I_r
\end{pmatrix}
\,.
\]

We denote by $\|\cdot\|$ the spectral norm (largest singular value) of a matrix.
We define $\Bc_{p,q}$ to be the open unit ball of $\Mc_{p,q}$ with respect to this norm:
\begin{align*}
\Bc_{p,q} := \big\{ X \in \Mc_{p,q} \mid \|X\|  < 1 \big\} \,.
\end{align*}

\begin{lemma}
\label{propositionPhiProperties}
The map $\phi_{p,q}$ from
$\Mc_{p,q}$ to  $\Bc_{p,q}$
defined by :
\begin{align*}
\phi_{p,q}(X)
= \big( I_q + XX^T \big) ^{-1/2} X
\end{align*}
is a bijection, with inverse
\begin{align*}
\psi_{p,q}(Y)= 
\big( I_q - YY^T \big) ^{-1/2} Y \enspace .
\end{align*}
Moreover, 
\begin{align*}
\phi_{p,q}(X) = X \big(I_p + X^TX\big) ^{-1/2}
\quad 
\text{and}
\quad
\big(\phi_{p,q}(X)\big)^T = \phi_{q,p}(X^T)
 \enspace .
\end{align*}
\end{lemma}
\begin{proof}
Let $X=UDV^T$ denote the singular value decomposition of $X$, so that $U,V$
are orthogonal matrices, and $D$ is a matrix consisting of a diagonal
block and a zero block. Then,
$\phi_{p,q}(X)= U\phi_{p,q}(D)V^T$, and a similar property
holds for the map
$\psi_{p,q}$.  Therefore, it suffices to check that $\psi_{p,q} \comp \phi_{p,q}  (X)= X$ when $X=D$, which is straighforward. By symmetry, we obtain
that $\phi_{p,q}  \comp \psi_{p,q}  (Y)=Y$ holds for all $Y$. The other properties are proved similarly.
\end{proof}

\subsection{Statement of the problem and the theorem}

As stated in Theorem~\ref{thm:main_theo}, the kernels $\Ker (A-C)$ and $\Ker (B-C)$ are central to the characterization of maximal lower bounds. 
In the following, we investigate the problem of the selection of a maximal lower bound of two symmetric matrices where subspaces of those kernels have been predetermined.
When $x^TCx > 0$, the line $\R x$ meets the surface $\{ x \mid x^TCx = 1 \}$ at two opposite points. Moreover, if $x \in \Ker (A-C)$, the surfaces $\{ x \mid x^TCx = 1 \}$ and $\{ x \mid x^TAx = 1 \}$ are tangent at those points.
When $x^TCx \leq  0$, it may be interpreted as a tangency at $\infty$.
For this reason, constraints on the kernels are called \emph{tangency constraints}.

Finally, following Theorem~\ref{thm:main_theo}, the dimension of the kernel of $A-B$ does not influence the structure of the set of maximal lower bound of $A$ and $B$. Thus, we assume that a reduction has been done and state Problem~\ref{prob:cons} and Theorem~\ref{thm:sec_theo} accordingly.

\begin{problem}[Maximal lower bounds with tangency constraints]
\label{prob:cons}
Let $A,B \in \Sc_n$ and let $(p,q,0)$ denote the inertia of $A-B$. Let $\Uc,\Vc$ be subspaces of $\R^n$.
We wish to find in $C \in \Sc_n$:
\begin{align*}
\begin{cases}
C\;\text{is a maximal lower bound of}\; A,B \\
\forall u \in \Uc, \; Cu = Bu \\
\forall v \in \Vc, \; Cv = Av \\
\end{cases}
\end{align*}
\end{problem}

Our second main result gives conditions for Problem~\ref{prob:cons} to have a solution and, if these conditions are met, a parametrization of the set of solutions. It shows that the set of solutions is of dimension $(p-\dim \Uc)(q-\dim \Vc)$, so that the problem has a unique solution if and only if one of the subspaces has maximal dimension.

\begin{theorem}
\label{thm:sec_theo}
Problem~\ref{prob:cons} has a solution if and only if
\begin{itemize}
	\item $A-B$ is positive definite  over $\Uc$
	\item $A-B$ is negative definite over $\Vc$
	\item $\Uc$ and $\Vc$ are orthogonal with respect to the indefinite scalar product $A-B$
\end{itemize}
If these conditions are met, then  
the set of solutions can be parametrized 
as in~\eqref{it:4} of Theorem~\ref{thm:main_theo},
with 
\[
	M \in \phi_{p,q}^{-1} ( \Bc_{p,q} \cap \mathcal{W} ) 
\]
where $\Wc$ is the affine subspace of $\Mc_{p,q}$ defined by
\begin{equation}
\label{feasibilityConditions}
	R \in \Wc \iff
	\begin{cases}
		\forall u \in \Uc,\, R^T\pi_p(u) + \pi_q(u) = 0 \\
		\forall v \in \Vc,\, R\pi_q(v) + \pi_p(v) = 0
	\end{cases}  \,.
\end{equation}
As soon as the conditions above are met, the intersection $\Bc_{p,q} \cap \mathcal{W}$ is nonempty.
The subspace $\mathcal{W}$ has dimension $ (p-\dim \Uc)(q-\dim \Vc)$, so that the solution is unique if and only if
\[
\dim \Uc = p \;\text{or}\; \dim \Vc = q \,.
\]
\end{theorem}

\begin{remark}
When $\Uc$ and $\Vc$ have maximal dimension, $\Vc$ is the orthogonal complement of $\Uc$ with respect the indefinite form $A-B$. 
Thus Theorem~\ref{thm:sec_theo} establishes a bijective correspondance between maximal lower bounds of $A,B$ and $p$-dimensionnal subspaces over which the matrix $A-B$ is positive definite.
In this way,
the set of maximal lower bounds is parametrized by an open semi-algebraic subset of 
the Grassmannian $\Gr(n,p)$.
\end{remark}

\subsection{Preliminary lemmas}

Before proving Theorem~\ref{thm:sec_theo}, we prove two useful results. First, Lemma~\ref{lem:Jpq_neq_def} shows that when the matrix $J_{p,q}$ is negative definite over a subspace $\Vc$, then there is a contractive mapping from the last $q$ coordinates of any vector of $\Vc$ to its first $p$ coordinates.

\begin{lemma}
\label{lem:Jpq_neq_def}
Let $\Vc$ be a subspace of $\R^n$ over which $J_{p,q}$ is negative definite, with $p+q = n$.
There is a matrix $R \in \Mc_{p,q}$ with $\norm{R} < 1$ such that :
\[
\forall x \in \Vc, \; \pi_p(x) = R \pi_q(x) \,.
\]
\end{lemma}

\begin{proof}
First, we show that the map $R_c$ from $\R^q$ to $\Vc$ defined by
\[
R_c : z \mapsto v \quad\text{s.t.}\quad v \in \Vc \quad\text{and}\quad \pi_q(x) = v
\]
is well defined.
Let $v,w \in \Vc$ such that $\pi_q(v) = \pi_q(w)$. We have $v-w \in \Vc$, thus $(v-w)^TJ_{p,q}(v-w) \leq 0$. This is rewritten as $\| \pi_p(v) - \pi_p(w) \|_2 \leq \| \pi_q(v) - \pi_q(w) \|_2 = 0$, hence $\pi_p(v) = \pi_q(w)$, which implies $v=w$.
The map $R_c$ is linear as its inverse map is the restriction of the linear map $\pi_q$ to $\Vc$.

Then, we define the linear map $R$ from $\R^q$ to $\R^p$ by $R(x) := \pi_p \circ R_c (x) $ on $\pi_q(\Vc)$ and we choose $R$ to be zero on ${\pi_q(\Vc)}^\perp$.
By definition of $R_c$ and $R$, we have for all $x \in \R^n$, $R \pi_q(x) = \pi_p(x)$.
The matrix $J_{p,q}$ is negative definite over $\Vc$, meaning that $\| \pi_p(x) \|_2 < \| \pi_q(x) \|_2$ when $x \in \Vc$ is nonzero. This implies that the map $R$ is a contraction on $\pi_q(\Vc)$:
\[
\| R \pi_q(x) \|_2 = \| \pi_p(x) \|_2 < \| \pi_q(x) \|_2 \,.
\]
As $R$ is zero on $\pi_q(\Vc)^\perp$, the map $R$ is a contraction on $\R^q$: $\norm{R} < 1$.
\end{proof}

Then, we solve Problem~\ref{prob:cons} in the easiest case, when the subspaces are $\Uc = \{ 0\}$ and $\Vc = \R x$, for $x \in \R^n$.
Since the proposition does not change if $r \neq 0$, we give its statement in the most general case.

\begin{proposition}
\label{problemSingleConstraint}
Let $A,B \in \Sc_n$ and $x \in \R^n$. Then, 
there exists a maximal lower bound $C$ of $A$ and $B$ such that $Ax = Cx$ if and only if $x^TAx < x^TBx$ or $Ax = Bx$.
\end{proposition}

\begin{proof}
Without loss of generality, we may assume that $A-B = J_{p,q,r}$.\\
($\implies$):\\
Assume that $C$ is a maximal lower bound of $A$ and $B$ such that $Ax = Cx$.
The constraint $Ax = Cx$ implies that $x^TAx = x^TCx \le x^TBx$ holds. We shall thus show that if $x^TAx = x^TBx$, then $Ax = Bx$.
Using Theorem~\ref{thm:main_theo}, there is $M \in \Mc_{p,q}$ such that
\[
C = A -
\begin{pmatrix}
I+MM^T & (I+MM^T)^{1/2}M \\
M^T(I+MM^T)^{1/2} & M^TM
\end{pmatrix}
\oplus 0_r
\,.
\]
The condition $Ax = Cx$ is rewritten as
\begin{align*}
\label{phi_pi_pi_constraint}
\phi_{p,q}(M)\pi_q(x) = -\pi_p(x) \,.
\end{align*}
The function $\phi_{p,q}$ maps the matrix $M$ to an element in the open ball $\Bc_{p,q}$, so $\norm{\phi_{p,q}(M)} < 1$. It follows that $\|\pi_p(x)\|_2 =  \| \phi_{p,q}(M)\pi_q(x) \|_2 < \| \pi_q(x) \|_2$ if $\pi_q(x) \neq 0$.
However, the assumption $x^T(A-B)x = 0$ implies $\| \pi_p(x) \|_2 = \| \pi_q(x) \|_2$, thus $\pi_q(x) = 0$ and $\pi_p(x)= 0$. We conclude with $(A-B)x = \pi_p(x) - \pi_q(x) = 0$.\\
($\impliedby$):\\
If $Ax = Bx$, then, by Theorem~\ref{thm:main_theo}, every maximal lower bound satisfies $Ax = Cx$. If $x^T(A-B)x < 0$, then $\| \pi_p(x) \|_2 < \| \pi_q(x) \|_2$. It is easily seen that using
\[
M = \phi_{p,q}^{-1} \Bigg( \frac{-\pi_p(x)\pi_q(x)^T}{{\|\pi_q(x)\|_2}^2} \Bigg)
\]
in the characterization in Theorem~\ref{thm:main_theo} provides a solution satisfying $Ax = Cx$.
\end{proof}

\subsection{Proof of Theorem~\ref{thm:sec_theo}}
\subsubsection*{Feasibility ($\implies$)}
~\\
Given a solution $C$ to Problem~\ref{prob:cons}, we have for $u \in \Uc$, $u^T(A-B)u = u^T(A-C)u - u^T(B-C)u$, where the first term is nonnegative and the second is zero. Hence $A-B$ is nonnegative over $\Uc$. For $v \in \Vc$, the reverse holds and $A-B$ is nonpositive over $\Vc$.
Moreover, if we have $x^T(A-B)x = 0$ for some $x \in \Uc \cup \Vc$, then by Proposition~\ref{problemSingleConstraint}, we have $(A-B)x = 0$ and $x=0$ as $A-B \in GL_n$. This shows that $A-B$ is positive definite over $\Uc$ and negative definite over $\Vc$.

Finally, for $u \in \Uc$ and $v \in \Vc$, as $\Uc \subseteq \Ker B-C$ and $\Vc \subseteq \Ker A-C$, we have $u^T(A-B)v = u^T(A-C)v - u^T(B-C)v = 0$, so $\Uc$ and $\Vc$ are orthogonal with respect to $A-B$.

\subsubsection*{Feasibility ($\impliedby$)}
~\\
We will use the characterization~\eqref{it:4} in Theorem~\ref{thm:main_theo} to build a solution to Problem~\ref{prob:cons}. Without loss of generality, we may assume that we work in a basis of $\R^n$ revealing the inertia of $A-B = J_{p,q}$.
Furthermore, we may assume that $\dim \Uc = p$ and $\dim \Vc = q$.
If this is not the case, let $\Uc_0$ denote a subspace of $\big[ (A-B) \cdot \Uc\big]^\perp$  over which $A-B$ is positive definite that has maximal dimension. Then let $\Vc_0$ denote a subspace of $\big[ (A-B) \cdot \Vc\big]^\perp \cap \big[ (A-B) \cdot (\Uc \oplus \Uc_0)\big]^\perp$ over which $A-B$ is negative definite that has maximal dimension.
The subspaces $\Uc \oplus \Uc_0$ and $\Vc \oplus \Vc_0$ then satisfy the assumptions.
We will prove that there is a matrix $R$ of size $p \times q$ satisfying $\norm{R} < 1$ and
\begin{align*}
u \in \Uc \iff \pi_q(u) = -R^T\pi_p(u) \,,\quad
v \in \Vc \iff \pi_p(v) = -R\pi_q(v)
\,.
\end{align*} 
The proof is done in two steps. First, we build a matrix $R$ satisfying the second and third equivalences using Lemma~\ref{lem:Jpq_neq_def}.
The matrix $J_{p,q}$ is negative definite over $\Vc$, so that for all nonzero $x\in\Vc$, we have $\norm{\pi_q(x)}_2^2 > \norm{\pi_p(x)}_2^2 \geq 0$, so $\pi_q(x) \neq 0$ and $\pi_q(\Vc) = \R^q$.
Then, we use the orthogonality condition to show that the first equivalence holds. Let $u \in \Uc$ and $v \in \Vc$. We have
\begin{align*}
\pi_q(v)^T\big(R^T\pi_p(u) - \pi_q(u)\big) & = \pi_p(u)^T\pi_p(v) - \pi_q(u)^T\pi_q(v) \\
					& = u^TJ_{p,q}v \\
					& = 0 \,.
\end{align*}
Hence $R^T\pi_p(u) - \pi_q(u) \in \R^q$ is orthogonal to $\pi_q(\Vc) = \R^q$, and is thus zero. It now suffices to take $M = \phi_{p,q}^{-1}(R)$ to build a solution to Problem~\ref{prob:cons} using~\eqref{it:4} in Theorem~\ref{thm:main_theo}.

\subsubsection*{Parametrization}
~\\
Let $C$ be a solution of Problem~\ref{prob:cons}. According to Theorem~\ref{thm:main_theo}, we can associate with $C$ a unique $M \in \Mc_{p,q}$. 
Given vectors $u \in \Uc$ and $v \in \Vc$, the constraints $Av = Cv$ and $Bu = Cu$ can be rewritten as
\begin{align*}
\phi_{p,q}(M)^T\pi_p(u) = -\pi_q(u) \;\;\text{and}\;\; \phi_{p,q}(M)\pi_q(v) = -\pi_p(v) \,.
\end{align*}
Moreover, we have $\norm{\phi_{p,q}(M)} < 1$, so that $\phi_{p,q}(M) \in \mathcal{W} \cap \Bc_{p,q}$.

Conversely, one checks easily that any solution $R$ of~\eqref{feasibilityConditions} provides a solution, as long as $R \in \Bc_{p,q}$. We have shown previously that as soon as the problem is feasible, the set $\mathcal{W} \cap \Bc_{p,q}$ is nonempty.

\subsubsection*{Dimension of $\mathcal{W}$}
~\\
Let $R \in \mathcal{W} \cap \Bc_{p,q}$. If $\dim \Uc \neq p$ and $\dim \Vc \neq q$, since $\dim \pi_p(\Uc) = \dim \Uc$ and $\dim \pi_q(\Vc) = \dim \Vc$, we can choose nonzero vectors $u_p \in {\pi_p(\Uc)}^\perp$ and $v_q \in {\pi_q(\Vc)}^\perp$. The ball $\Bc_{p,q}$ is an open set, thus for small enough positive $\epsilon$, the matrix $R' := R + \epsilon u_p{v_q}^T$ is also in $\Bc_{p,q}$ and  satisfies the equations~\eqref{feasibilityConditions}. The matrix $R'$ produces a different solution than $R$ since $R \neq R'$ and $\phi_{p,q}$ is a bijection, so that $\dim \mathcal{W} \geq \dim {\pi_p(\Uc)}^\perp \times \dim {\pi_p(\Vc)}^\perp = (p-\dim \Uc)(q-\dim \Vc)$.

If $R,R' \in \mathcal{W}$ are solutions of~\eqref{feasibilityConditions}, then we have
\begin{align*}
\forall u \in \Uc,\, (R-R')\pi_q(u) = 0 \qquad
\forall v \in \Vc,\, \pi_p(v)^T(R-R') = 0 \,
\end{align*}
which yields the reverse inequality $\dim \mathcal{W} \leq \dim {\pi_p(\Uc)}^\perp \times \dim {\pi_p(\Vc)}^\perp$.  \hfill $\qed$

\section{Examples}

\label{sec:appli}

We recall the definition of ellipsoids, the equivalence between the inclusion of ellipsoids and the L\"owner order and the algebraic counterpart of tangency between ellipsoids.
\begin{definition}
We denote by $\Qc_A$ the quadric associated with the \emph{symmetric} matrix $A$, defined by:
\begin{align*}
\Qc_A = \{ x \in \R^n \mid x^TAx \leq 1 \} \,.
\end{align*}
The set $\Qc_A$ is convex if and only if the matrix $A$ is positive semidefinite. 
Then, we call the set $\Qc_A$ an ellipsoid, and it will also be denoted by $\Ec_A$.
The set $\Ec_A$ is bounded if and only if the matrix $A$ is positive definite. Moreover, it always has a nonempty interior.
The inclusion of the ellipsoid $\Ec_A$ in the quadric $\Qc_B$ is equivalent to the positivity of the matrix $A-B$, meaning that the inclusion of ellipsoids in quadrics and the ordering of the corresponding matrices is equivalent, up to reversal:
\begin{align*}
\Ec_A \subseteq \Qc_B \iff B \preceq A \,.
\end{align*}
\end{definition}
This also means that, given positive definite matrices $A,B$, the quadric $\Qc_C$ associated with a maximal lower bound $C$ of $A$ and $B$ in the L\"owner order is a minimal upper bound for the ellipsoids $\Ec_A$ and $\Ec_B$, in the inclusion order.
\begin{remark}
In the general case,
\begin{align*}
\Qc_A \subseteq \Qc_B \notimplies B \preceq A \,,
\end{align*}
as shown with $A = 2 \oplus (-2)$ and $B = 1 \oplus (-1)$. For $(x,y)\in \Qc_A$, one clearly has $2x^2-2y^2 \leq 1 \leq 2$, which implies $(x,y) \in \Qc_B$. However, we have $A-B = 1 \oplus (-1) \not\succeq 0$.
\end{remark}

\subsection{In dimension 2: $\Oc(1,1)/\big( \Oc(1) \times \Oc(1) \big)$}

This case arises whenever two symmetric matrices $A$ and $B$ of order $2$ are not comparable. The maps $( X \mapsto X+\lambda I_n)_{\lambda \in \R}$ and $( X \mapsto U^TXU )_{U \in \GL_n}$ are all order-preserving isomorphisms. This implies that, given such an isomorphism $\phi$, the set of maximal lower bounds of $\phi(A)$ and $\phi(B)$ is exactly the image of the set of maximal lower bounds of $A$ and $B$ by the map $\phi$. Thus one can easily show that we may assume without loss of generality that
\begin{align*}
A = \begin{pmatrix}
2 & 0 \\ 0 & 1
\end{pmatrix}
\qquad
B = \begin{pmatrix}
1 & 0 \\ 0 & 2
\end{pmatrix} \,. 
\end{align*}
We have the explicit description of the set of hyperbolic isometries $\Oc(1,1)$:
\begin{align*}
\Oc(1,1) = \Bigg\{
\begin{pmatrix}
\epsilon_1 \ch \theta & \epsilon_2 \sh \theta \\
\epsilon_1 \sh \theta & \epsilon_2 \ch \theta
\end{pmatrix}
\mid \theta \in \R,\, \epsilon_1,\epsilon_2 \in \{-1,1\}
 \Bigg\} \,.
\end{align*}
The quotient set $\Oc(1,1)/\big( \Oc(1) \times \Oc(1) \big)$ is in this case equal to the classical set of hyperbolic rotations:
\begin{align*}
\Oc(1,1)/\big( \Oc(1) \times \Oc(1) \big)= \Bigg\{
\begin{pmatrix}
\ch \theta & \sh \theta \\
\sh \theta & \ch \theta
\end{pmatrix}
\mid \theta \in \R
 \Bigg\} \,.
\end{align*}
Note that in this special case, the quotient set has a group structure.

This gives us the parametrization of the minimal upper bounds $C_\theta$ of $A$ and $B$:
\begin{align*}
C_\theta = \begin{pmatrix}
2 - \ch^2 \theta & \ch \theta \sh \theta \\
\ch \theta \sh \theta & 2 - \ch^2 \theta
\end{pmatrix} \,.
\end{align*}
Moreover,
the tangency subspaces to $\Ec_A$ and $\Ec_B$ are respectively equal to
$\R \begin{pmatrix} \sh \theta & -\ch \theta \end{pmatrix}^T$ and $\R \begin{pmatrix} \ch \theta & -\sh \theta \end{pmatrix}^T$.
This is depicted in Figure~\ref{fig:selection_dim2}.

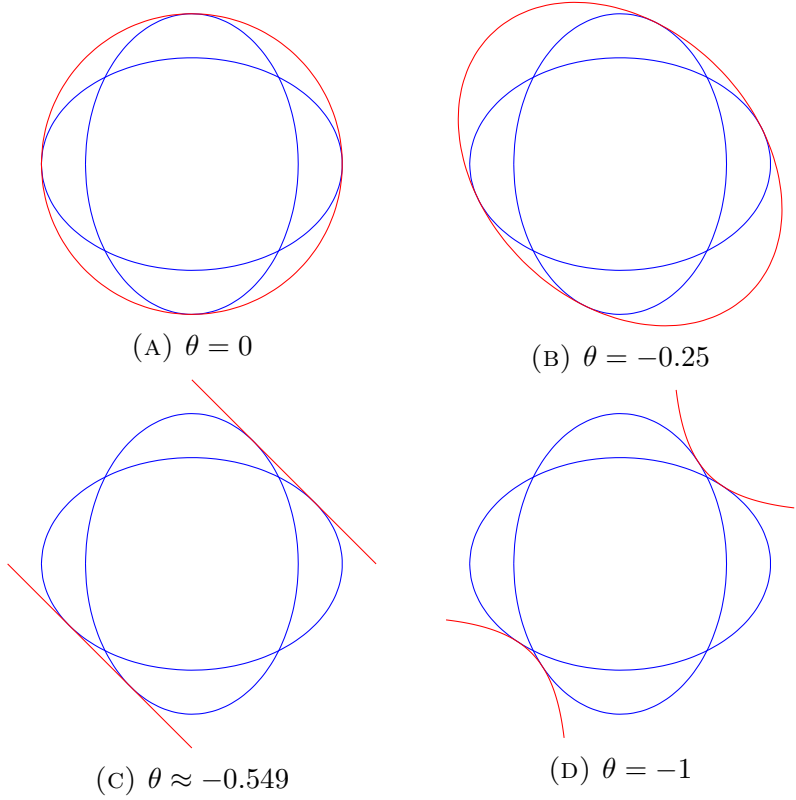
\begin{figure}[!t]
\centering
\begin{subfigure}{0.45\textwidth}
\centering
\begin{tikzpicture}[scale=2]
\draw [blue, domain=-pi:pi, samples=50, smooth] plot ({cos(\x r)},{sin(\x r)/sqrt(2)});
\draw [blue, domain=-pi:pi, samples=50, smooth] plot ({cos(\x r)/sqrt(2)},{sin(\x r)});
\draw [red, domain=-pi:pi, samples=50, smooth] plot 
({cos(\x r)},
 {sin(\x r)});
\end{tikzpicture}
\caption{$\theta = 0$}
\end{subfigure}%
\begin{subfigure}{0.45\textwidth}
\centering
\begin{tikzpicture}[scale=2]
\draw [blue, domain=-pi:pi, samples=50, smooth] plot ({cos(\x r)},{sin(\x r)/sqrt(2)});
\draw [blue, domain=-pi:pi, samples=50, smooth] plot ({cos(\x r)/sqrt(2)},{sin(\x r)});
\draw [red, domain=-pi:pi, samples=50, smooth] plot 
({0.6463772544*cos(\x r) + 0.8602556420*sin(\x r)},
 {0.6463772544*cos(\x r) - 0.8602556420*sin(\x r)});
\end{tikzpicture}
\caption{$\theta = -0.25$}
\end{subfigure}
\begin{subfigure}{0.45\textwidth}
\centering
\begin{tikzpicture}[scale=2]
\draw [blue, domain=-pi:pi, samples=50, smooth] plot ({cos(\x r)},{sin(\x r)/sqrt(2)});
\draw [blue, domain=-pi:pi, samples=50, smooth] plot ({cos(\x r)/sqrt(2)},{sin(\x r)});
\draw [red, domain=0:sqrt(3)/sqrt(2)] plot ({\x}, {-\x + sqrt(3)/sqrt(2)});
\draw [red, domain=-sqrt(3)/sqrt(2):0] plot ({\x}, {-\x - sqrt(3)/sqrt(2)});
\end{tikzpicture}
\caption{
$\theta \approx -0.549$}
\end{subfigure}%
\begin{subfigure}{0.45\textwidth}
\centering
\begin{tikzpicture}[scale=2]
\draw [blue, domain=-pi:pi, samples=50, smooth] plot ({cos(\x r)},{sin(\x r)/sqrt(2)});
\draw [blue, domain=-pi:pi, samples=50, smooth] plot ({cos(\x r)/sqrt(2)},{sin(\x r)});
\draw [red, domain=-0.75:0.75, samples=50, smooth] plot 
({0.5340780207*exp(\x) + 0.0567527431*exp(-\x)},
 {0.0567527431*exp(\x) + 0.5340780207*exp(-\x)});
\draw [red, domain=-0.75:0.75, samples=50, smooth] plot 
({-0.5340780207*exp(\x) - 0.0567527431*exp(-\x)},
 {-0.0567527431*exp(\x) - 0.5340780207*exp(-\x)});
\end{tikzpicture}
\caption{$\theta = -1$}
\end{subfigure}
\caption{Minimal quadrics $\Qc_\theta$ (in red) associated with $\Ec_A$ and $\Ec_B$ (in blue) for various values of $\theta$.}
\label{fig:selection_dim2}
\end{figure}
%

\subsection{The quotient Lorentz set: $\Oc(n,1)/\big( \Oc(n) \times \Oc(1) \big), \; n\geq 2$ }
\label{sec:ex2}

Following Lemma~\ref{thm:gallier}, the set $\Oc(n,1)/\big( \Oc(n) \times \Oc(1) \big)$ can be identified to $\R^n$ via the bijection $\phi := \phi_{n,1}$ defined by
\begin{align*}
\phi : {\mat}\mapsto
\begin{pmatrix}
(I_n + {\mat}{\mat}^T)^{1/2} & {\mat} \\
{\mat}^T & \sqrt{1+{\mat}^T{\mat}}
\end{pmatrix} \,.
\end{align*}

In this case, when $pq = n > 1$, the quotient set does not have a group structure.
Let $(e_i)_{1\leq i\leq n}$ denote the canonical base of $\R^n$.
The product $M := \phi(e_1)\phi(e_2)$ can be computed explicitly and it is not even symmetric: we have $M_{2,1} = 0$ whereas $M_{1,2} = 1$.

We shall illustrate the results of Theorem~\ref{thm:sec_theo} on an example with $p = 2$ and $q = 1$, with the matrices
$A = 2 \oplus 2 \oplus 1$ and $B = 1 \oplus 1 \oplus 2$,
so that $A-B = J_{2,1}$.
Theorem~\ref{thm:main_theo} states that the set of maximal lower bounds of $A$ and $B$, denoted $\mlb_{A,B}$, has dimension $2$ and its elements $C_{\mat}$ are given, for ${\mat}\in \R^2$ by
\begin{align*}
C_{\mat} = 
A - 
\begin{pmatrix}
I_2 + {\mat}{\mat}^T & (I_2 + {\mat}{\mat}^T)^{1/2}{\mat} \\
{\mat}^T(I_2 + {\mat}{\mat}^T)^{1/2} & {\mat}^T{\mat}
\end{pmatrix}
 \,.
\end{align*}
For all ${\mat}\in \R^2$, we also have $\dim \Ker (A-C_{\mat}) = 1$ and $\dim \Ker (B-C_{\mat}) = 2$.

Let $v = (x\;0\;z)^T$ denote some non-zero vector. We shall solve Problem~\ref{prob:cons} in the cases where $(\Uc, \Vc) = (\R v, \{0\})$ and $(\Uc, \Vc) = (\{0\}, \R v)$.

\subsubsection*{Case 1: $\Uc = \R v$ and $\Vc = \{0\}$.}
~\\
In this case, we have $p \neq \dim \Uc$ and $q \neq \dim \Vc$, so by Theorem~\ref{thm:sec_theo} the set of solutions is not reduced to a point.
The problem has a solution if and only if $x^2 > z^2$, and the solutions are parametrized by the contractive elements of the affine subspace $\Wc$ of $\Mc_{2,1}$ defined by
$ R \in \Wc$ if and only if  $R^T (x\,0)^T + z = 0$. Denoting $r = -z/x$, so that $|r| < 1$, we have
\begin{align*}
\Wc =  \big\{ R_t := (r\;t)^T \mid t \in \R \big\} \,.
\end{align*}
Moreover, we have $\norm{R_t}^2 = r^2+t^2$ so that, since $r^2 < 1$, the set $\Wc \cap \Bc_{p,q}$ is non-empty. Then, for $|t| < \sqrt{1-r^2}$, we recover the matrix \begin{align*}
{\mat} = \phi_{p,q}^{-1}(R_t) = (1-r^2-t^2)^{-1/2} \begin{pmatrix}
r \\ t
\end{pmatrix} \,.
\end{align*}
Finally, we get the parametrization of the kernels:
\begin{align*}
\Ker (A-C_{\mat}) & = \Span \Big\{
\begin{pmatrix}
z & -tx & x
\end{pmatrix}^T \Big\}\,,
\\
\Ker (B-C_{\mat}) & = \Span \Big\{
\begin{pmatrix}
x & 0 & z
\end{pmatrix}^T
\;,\;
\begin{pmatrix}
txz & x^2+y^2 & -x^2t
\end{pmatrix}^T
\Big\} \,.
\end{align*}
The set of solutions is parametrized by a single real parameter $t$ as expected from Theorem~\ref{thm:sec_theo}.

\subsubsection*{Case 2: $\Uc = \{0\}$ and $\Vc = \R v$.}
~\\
In this case, we have $q = \dim \Vc$, so the solution is unique.
Indeed, the problem has a solution if and only if $x^2 < z^2$ and the affine subspace $\Wc$ of $\Mc_{2,1}$ is reduced to the point $R := (-x/z\,,0)$, which satisfies $\norm{R} < 1$.
Figure~\ref{fig:selection_dim3} depicts several minimal quadrics associated with the ellipsoids $\Ec_A$ and $\Ec_B$.

\begin{figure}[!t]
\centering
\begin{subfigure}{0.45\textwidth}
\centering
\includegraphics[width=\textwidth,trim={12cm 12cm 12cm 12cm},clip]{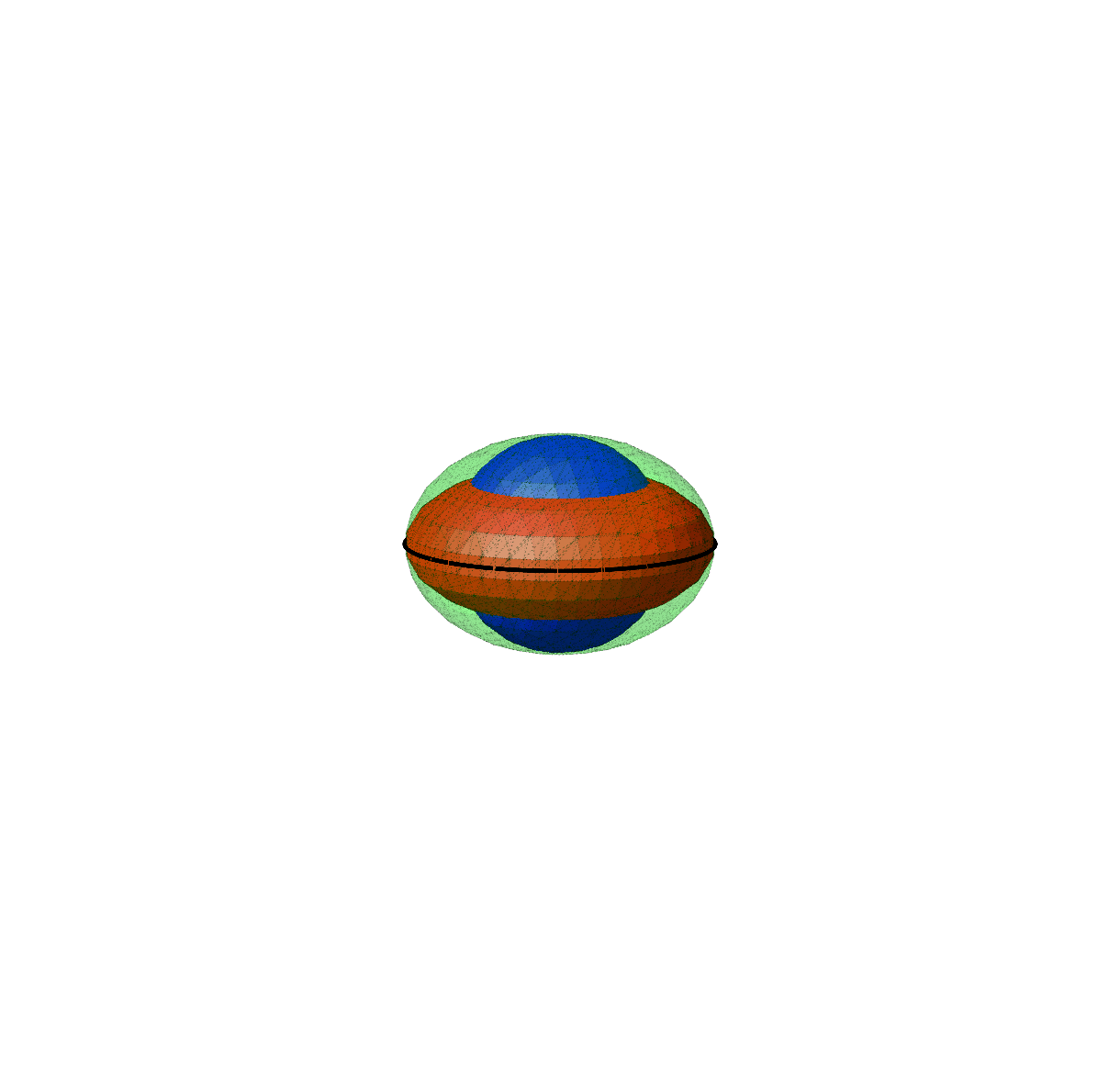}
\caption{${\mat} = (0 \;,0)$}
\end{subfigure}%
\begin{subfigure}{0.45\textwidth}
\centering
\includegraphics[width=\textwidth,trim={12cm 12cm 12cm 12cm},clip]{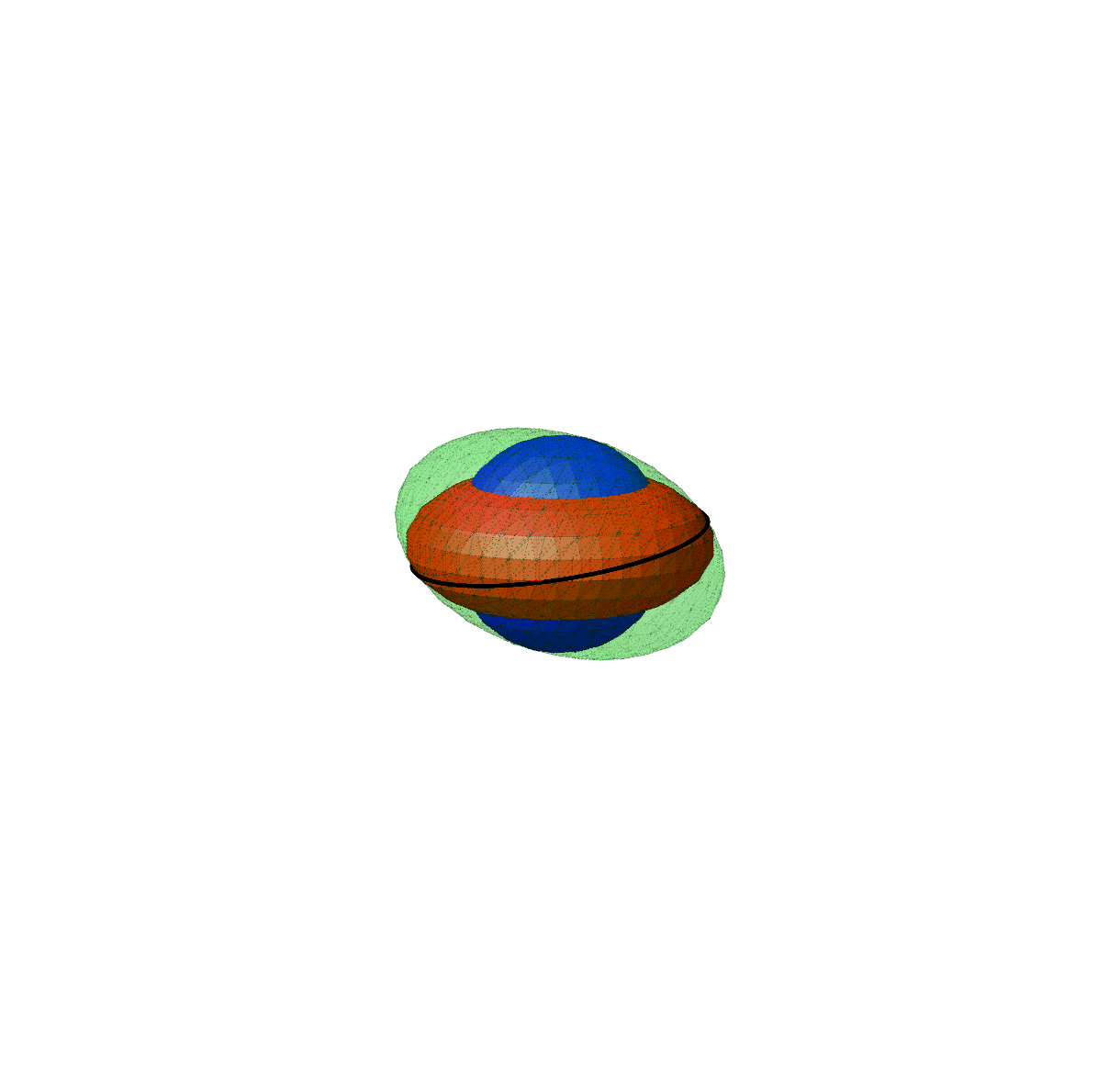}
\caption{${\mat} = (0 \;,0.25)$}
\end{subfigure}
\begin{subfigure}{0.45\textwidth}
\centering
\includegraphics[width=\textwidth,trim={12cm 12cm 12cm 12cm},clip]{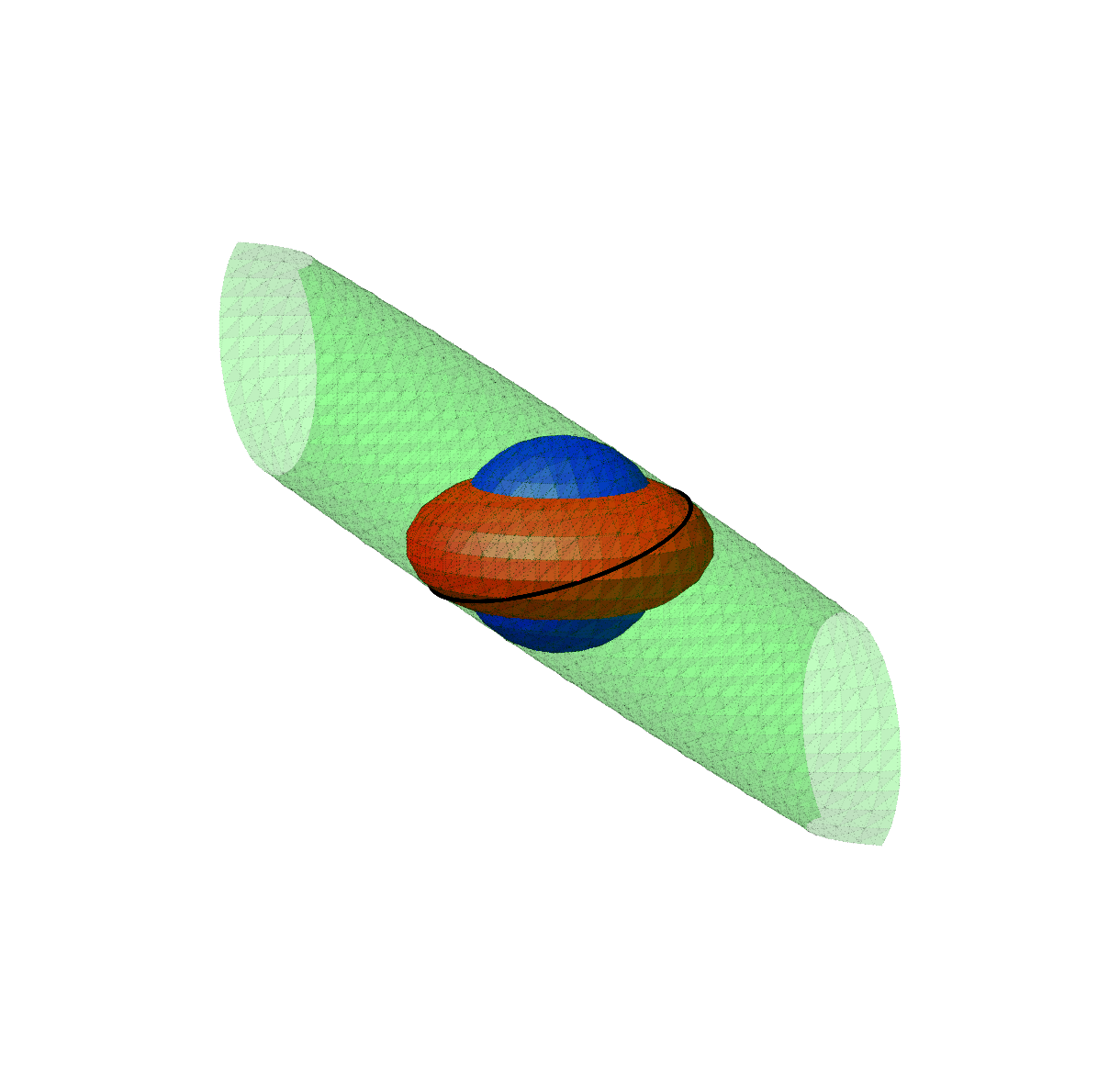}
\caption{${\mat} = (0 \;,0.549)$}
\end{subfigure}%
\begin{subfigure}{0.45\textwidth}
\centering
\includegraphics[width=\textwidth,trim={12cm 12cm 12cm 12cm},clip]{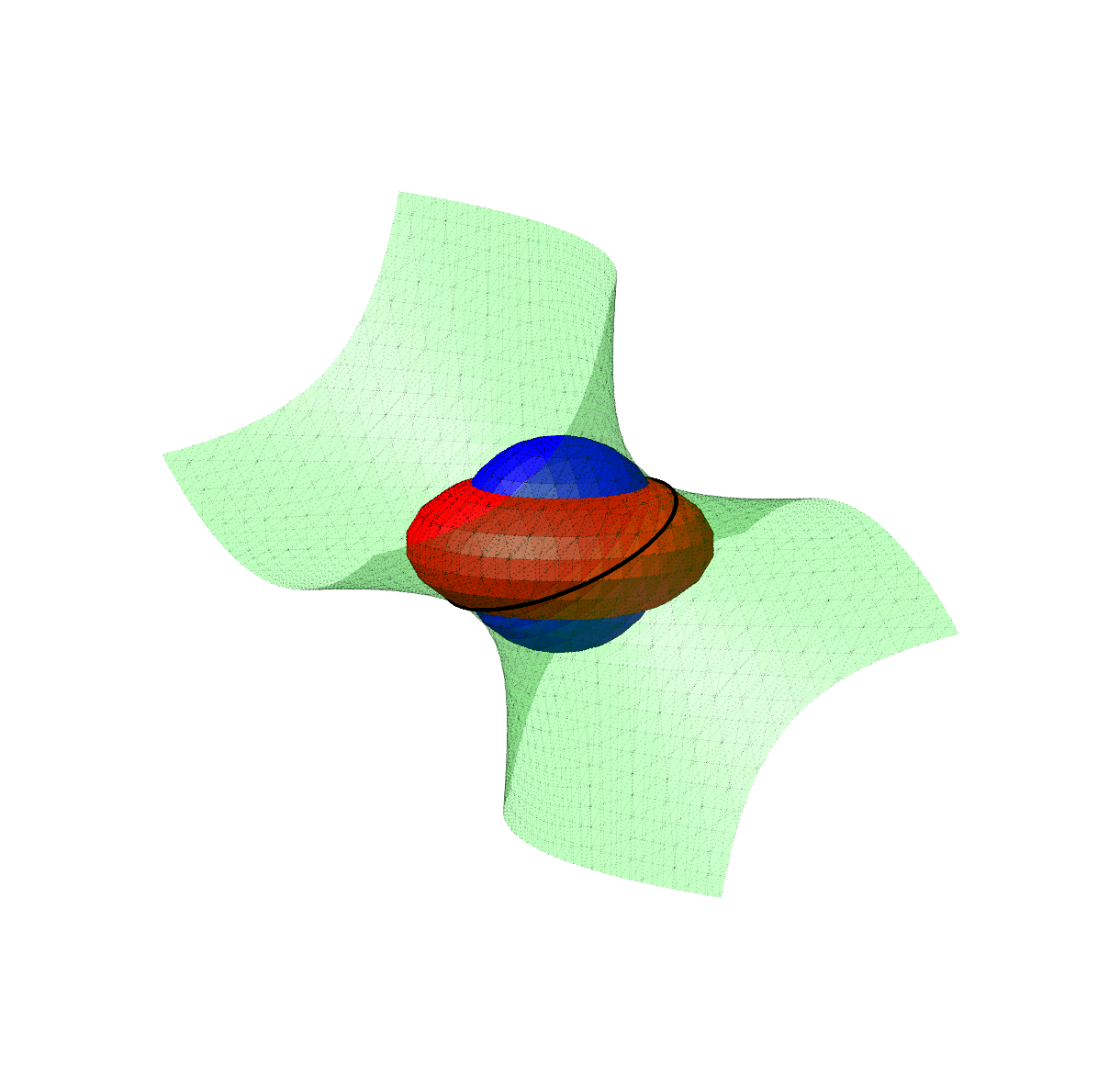}
\caption{${\mat} = (0 \;,1)$}
\end{subfigure}
\caption{Minimal quadrics $\Qc_{C_{\mat}}$ (in green) associated with $\Ec_A$ and $\Ec_B$ (in blue and red) for various values of ${\mat}$. The black line shows the tangency points between the quadrics $\Ec_B$ and $\Qc_{C_{\mat}}$.}
\label{fig:selection_dim3}
\end{figure}

%

\section*{Acknowledgements}
The author thanks St\'ephane Gaubert for pointing out the problem of parametrization of maximal lower bounds. The author also thanks him with Xavier Allamigeon for many useful remarks and suggestions on previous versions of this work. Finally, the author thanks Peter Semrl for pointing out the link with Kadison's result.

\bibliographystyle{amsplain}
\bibliography{loewner}

\end{document}